\documentclass[a4paper,10pt]{scrartcl}

\usepackage[fleqn]{amsmath}
\usepackage{amssymb}
\usepackage{amsthm}
\usepackage{stmaryrd}
\usepackage{bbm}
\usepackage{hyperref}
\usepackage{mathtools}

\mathtoolsset{showonlyrefs}

\DeclareMathOperator{\Var}{Var}

\DeclareMathOperator{\supp}{supp}

\newcommand{\bC}{\ensuremath{\mathbb{C}}}

\newcommand{\bE}{\ensuremath{\mathbb{E}}}

\newcommand{\bP}{\ensuremath{\mathbb{P}}}

\newcommand{\bZ}{\ensuremath{\mathbb{Z}}}

\newcommand{\ind}{\ensuremath{\mathbbm{1}}}

\newcommand{\cA}{\ensuremath{\mathcal{A}}}
\newcommand{\cB}{\ensuremath{\mathcal{B}}}

\newcommand{\bs}{\backslash}

\newcommand{\abs}[1]{\left\vert \, #1 \, \right\vert}
\newcommand{\norm}[1]{\left\Vert \, #1 \, \right\Vert}

\newcommand{\ddx}[1][1]{\ifnum#1=1 \frac{d}{dx} \else \frac{d^{#1}}{dx^{#1}} \fi}
\newcommand{\ddy}[1][1]{\ifnum#1=1 \frac{d}{dy} \else \frac{d^{#1}}{dy^{#1}} \fi}
\newcommand{\ddt}[1][1]{\ifnum#1=1 \frac{d}{dt} \else \frac{d^{#1}}{dt^{#1}} \fi}

\newtheorem{theorem}{Theorem}[section]
\newtheorem{lemma}[theorem]{Lemma}
\newtheorem{proposition}[theorem]{Proposition}

\begin{document}

\title{The super-critical contact process has a spectral gap} 
\author{Florian V\"ollering\thanks{University of G\"ottingen, IMS, Goldschmidtstra\ss e 7
37077 G\"ottingen,
Germany \newline email: florian.voellering@mathematik.uni-goettingen.de, tel.: +49 551 3913520}}
\date{}
\maketitle

\begin{abstract}
  We consider the super-critical contact process on $\bZ^d$. It is known that measures which dominate the upper invariant measure $\mu$ converge exponentially fast to $\mu$. However, the same is not true for measures which are below $\mu$, as the time to infect a large empty region is related to its diameter. The result of this paper is the existence of a spectral gap in $L^2(\mu)$, that is, the spectrum of the generator is empty inside an open strip $\{z\in\bC: -\lambda<\Im(z)<0\}$ of the complex plane. This is equivalent to the fact that the variance of the semi-group of the contact process decays exponentially fast. It is perhaps surprising that the existence of the spectral gap has not been proven before. One of the reasons is that the contact process is non-reversible, and hence many methods from spectral theory are not applicable. 
\end{abstract}

{{\bf MSC classes:} 60K35 (primary), 82C22 (secondary)}


\section{Introduction}
The contact process on $\bZ^d$ is a well-studied process. It is typically framed as a simple model for infections, where healthy individuals can be infected by their neighbours, and infected individuals recover at a constant rate. Depending on infection and recovery rates there is a phase transition. If the infection rate is larger than a critical value, then infections can survive indefinitely, and if it is smaller than the critical value, then infections vanish exponentially fast. 

Let $\Omega=\{0,1\}^{\bZ^d}$. For $\eta\in\Omega$ and $x\in\bZ^d$, we denote by $\eta^x$ the configuration $\eta$ flipped at $x$, meaning that $\eta^x$ is identical to $\eta$ except at site $x$, where a 1 is replaced by a 0 and vice versa. By slight abuse of notation we also write $\eta = \{ x\in\bZ^d: \eta(x)=1\}$. We say a site $x\in\bZ^d$ in a configuration $\eta\in\Omega$ is infected if $x\in\eta$.

The contact process with infection rate $\lambda>0$ and recovery rate 1 is the Markov process with generator
\begin{align}
 L f (\eta) = \sum_{x\in\bZ^d}(1-\eta(x))\lambda \sum_{y\sim x} \eta(y)[f(\eta^x)-f(\eta)] + \sum_{x\in\bZ^d} \eta(x)[f(\eta^x)-f(\eta)].
\end{align}
Here $\sum_{y\sim x}$ denotes the summation over nearest neighbours of $x$. For a subset $A\subset \bZ^d$ we denote by $(\xi_t^A)_{t\geq0}$ the contact process started with exactly all sites in $A$ infected. If $A=\{x\}$, we simply write $\xi_t^x$. By $(P_t)_{t\geq0}$ we denote the semi-group generated by $L$.

The contact process exhibits a phase transition in $\lambda$ with critical value $\lambda_c$. 
If $\lambda>\lambda_c$, then the contact process started from a single infection survives with a positive probability, and if $\lambda\leq\lambda_c$ it goes extinct almost surely. We will only consider the super-critical case $\lambda>\lambda_c$. In this case, there is a unique non-trivial invariant ergodic measure, the upper invariant measure, which we denote by $\mu$. The variance of a function $f:\Omega\to\bC$ with respect to $\mu$ is 
\[ \Var_\mu(f) = \int \abs{f-\int f\,d\mu}^2d\mu. \]

Before going into more detail let us state the main theorem
\begin{theorem}
 There exists an $\alpha>0$ so that the set $\{z\in\bC : -\alpha/2 < \Re(z)<0 \}$ is contained in the resolvent set of the generator $L$ in $L^2(\mu)$. Equivalently, for all $f\in L^2(\mu), t\geq0$,
\begin{align}\label{eq:L2-decay}
 \Var_\mu(P_t f) \leq e^{-\alpha t}\Var_\mu(f).
\end{align}
\end{theorem}
The proof of the theorem will be done in three steps.  After a brief introduction to the graphical construction in Section \ref{section:graphical} we will show in Section \ref{section:Zd} that $\Var_\mu( P_t f)$ decays exponentially fast with rate $-\alpha+o(1)$. However, in contrast to \eqref{eq:L2-decay} the right hand side has a different dependency on $f$ than via the variance of $f$. In Section \ref{section:finite} we will show that the same result holds for the contact process in a finite box with an infected boundary condition. On the finite box we can then obtain the spectral gap. Finally in Section \ref{section:gap} we lift the result for finite boxes back to the infinite lattice.

\section{Graphical construction}\label{section:graphical}
A widely used method for understanding the contact process is the graphical construction. Let $(N_t^{xy})_{t\geq0}$ be independent Poisson processes with rate $\lambda$, for each directed edge $xy$ of $\bZ^d$. The jump events of $N^{xy}$ correspond to infection events across the arrow connecting $x$ to $y$. That is if site $x$ is infected before the infection event, then afterwards $y$ is as well. Let $(N^x_t)_{t\geq0}$ be independent Poisson processes of rate 1, at each site $x\in\bZ^d$, whose jump times correspond to recovery events. After a recovery event at a site $x$ a possible infection of $x$ is cured. The value of the contact process $\xi_t^A$ can then be understood as a function of $(N^{xy}_s)_{0\leq s\leq t}$ and $(N^x_s)_{0\leq s \leq t}$, as infection events and recovery events determine the evolution of the initial set of infected sites $A$. See \cite{LIGGETT:99} for a more in-depth review of this construction. We will denote the law of the graphical construction by $\bP$, and the corresponding 
expectation by $\bE$. Hence we have $P_t f(\eta) = \bE f(\xi_t^\eta)$.

Besides giving a natural understanding of the dynamics the graphical construction has the advantage that it provides a natural coupling between contact processes. Two contact processes with different initial configurations can be coupled by using the same infection and recovery events for the two different initial sets of infections.

The contact process also satisfies a well-known self-duality relation, which is the following: For two subsets $A,B\subset \bZ^d$,
\begin{align}
 \bP(\xi^A_t \cap B \neq \emptyset ) = \bP(\xi^B_t\cap A\neq \emptyset).
\end{align}
This follows by reversing time and direction of arrows in the graphical construction.

Self-duality also gives insight into the upper invariant measure $\mu$. Let $A$ be a subset $A$ of $\bZ^d$. We say $A$ is not infected if all sites in $A$ are not infected. We say $\xi^A$ becomes extinct if $\xi^A_t=\emptyset$ for some $t\geq0$. Since 
\begin{align}
\lim_{t\to\infty}\bP(\xi^{\bZ^d}_t\cap A = \emptyset) = \mu(\{\text{$A$ is not infected}\}),                                                                               
\end{align}
we have
\begin{align}\label{eq:mu-duality}
\mu(\{\text{$A$ is not infected}\}) = \bP(\xi^A \text{ becomes extinct}).
\end{align}

\section{Estimates on \texorpdfstring{$\bZ^d$}{Z\^{ }d}}\label{section:Zd}
For complex variables $z\in\bC$, we write $\abs{z}^2=z\cdot\overline{z}$.
Let $f:\Omega\to\bC$. We write 
\begin{align}\label{eq:deltas}
 \delta_f(x) := \sup_{\eta\in\Omega}\abs{f(\eta^x)-f(\eta)}.
\end{align}
The main result of this section is the following estimate on the $L^2$-decay of the semi-group of the contact process.
\begin{proposition}\label{prop:var-decay}
There are constants $\alpha,C>0$ so that for any $f:\Omega\to\bC$ with $\sum_{x\in\bZ^d}\delta_f(x)^2<\infty$,
\begin{align}\label{eq:prop-var-decay}
\Var_\mu(P_tf)\leq C \int_t^\infty (s+1)^{3d} e^{-\alpha s}\,ds \sum_{x\in\bZ^d}\delta_f(x)^2.
\end{align}
\end{proposition}
Note how on the right hand side a different (semi-)norm than the $L^2$-norm appears. As the contact process is not reversible one cannot obtain the spectral gap from \eqref{eq:prop-var-decay} by simple spectral theory.

In this section we will prove a sequence of estimates aimed at proving Proposition \ref{prop:var-decay}. Often, various constants appear in the estimates. These constants are simply labelled $c_1, c_2, ...$. They are assumed to be positive, may depend on the model parameters, and are not assumed to be the same between different lemmas unless stated otherwise.

We will make use of the following well-known facts about the super-critical contact process. They can for example be found in \cite{LIGGETT:99}, Theorem 2.30.
\begin{lemma}\label{lemma:liggett}
 There are constants $c_1,c_2>0$ so that for any $A\subset \bZ^d$, $t\geq0$,
\begin{enumerate}
 \item $\bP (\xi^A \text{becomes extinct}) \leq c_1e^{-c_2\abs{A}}$;
\item $\bP( \xi^A \text{becomes extinct}, \xi^A_t\neq \emptyset) \leq c_1e^{-c_2 t}$.
\end{enumerate}
\end{lemma}

The first lemma we prove concerns itself with the growth of a surviving infection. To be more precise, we will prove that the probability of a single surviving infection not having grown to a size of at least order $t$ is exponentially small.
\begin{lemma}\label{lemma:ldp-survival}
There are constants $c_1,c_2,c_3>0$ so that
\begin{align}\label{eq:growth}
\bP\left( |\xi^y_t|<c_1t \;\middle|\; \xi^y \text{ survives} \right) \leq c_2 e^{-c_3 t}. 
\end{align}
\end{lemma}
\begin{proof}
By translation invariance, we can assume $y=0$.
 To prove the lemma, we use large deviations for the shape theorem and stochastic domination of a Bernoulli product measure.
Let 
\begin{align}
 &t(x) := \inf\{t\geq0 : \xi_t^0 (x)=1 \},	\\
&t'(x) := \inf\{t\geq 0: \xi_s^0(x)=\xi_s^{\bZ^d}\,\forall s\geq t\},	\\
&H(t)  := \{x\in\bZ^d : t(x)\leq t\},	\\
&K'(t) := \{x\in\bZ^d : t'(x)\leq t\},	\\
&B_r   := \{x\in\bZ^d : \norm{x}_\infty \leq r \}.
\end{align}
In \cite{GARET:MARCHAND:12}, Theorem 1.1 implies that there are constants $c_4,c_5, c_6>0$ so that
\begin{align}\label{eq:ldp-1}
 \bP\left( B_{c_4t}\not\subset H(t)\cap K'(t) \;\middle|\; \xi^0 \text{ survives} \right)\leq c_5e^{-c_6t}.
\end{align}
Note that for any $\eta\in\Omega$ with $0\in\eta$, $\xi^0\subset \xi^\eta \subset \xi^{\bZ^d}$. Define the good event $G_t:=\{\xi^0 \text{ survives}, B_{c_4t}\subset H(t)\cap K'(t)\}$. Conditioned on $G_t$ the law of the configuration inside the ball $B_{c_4 t}$ is absolutely continuous with respect to $\mu$: Let $\cA$ be an event which depends only on the sites in $B_{c_4 t}$. Then, by the definition of $G_t$ and time invariance of $\mu$,
\begin{align}
  \bP\left(\xi^0_t \in \cA \;\middle|\; G_t \right) &= \int\bP\left(\xi^\eta_t \in \cA \;\middle|\; G_t \right) \mu(d\eta\,|\,0\in\eta) \leq \frac{\mu(\cA)}{\bP(G_t)\mu(0\in\eta)}.\label{eq:density}
\end{align}
In \cite{LIGGETT:STEIF:06}, Corollary 4.1 states that the upper invariant measure $\mu$ of the super critical contact process stochastically dominates a Bernoulli product measure $\nu_\rho$ with some density $\rho>0$. Large deviations for Bernoulli product measures show that the probability of seeing only half as many ones as expected in $n$ trials is exponentially small in $n$. When this is applied to the number of infections in a ball, we get 
\begin{align}
 \mu\left( \abs{B_{c_4t}\cap \eta} < \frac\rho2 \abs{B_{c_4t}} \right)\leq \nu_\rho\left( \abs{B_{c_4t}\cap \eta} < \frac\rho2 \abs{B_{c_4t}} \right)\leq c_7e^{-c_8t}.
\end{align}
Hence, by \eqref{eq:density}, 
\begin{align}\label{eq:ldp-2}
 \bP\left( \abs{\xi^0_t\cap B_{c_4t}} <\frac\rho2 \abs{B_{c_4 t}} \;\middle|\; G_t \right) \leq \frac{c_7e^{-c_{8}t}}{\bP(G_t)\mu(0\in\eta)}. 
\end{align}
Combining \eqref{eq:ldp-1} and \eqref{eq:ldp-2} plus the fact that in the supercritical regime a single infection has a positive probability to survive yields the claim.
\end{proof}

We will now prove that a single discrepancy in the initial configuration will typically vanish quickly and not grow to a large size. 
\begin{lemma}\label{lemma:discrepancy}
There are constants $c_1,c_2,c_3,\alpha>0$ so that
\begin{align}
\int \bP\left(\xi^\eta_t(y)\neq\xi^{\eta^x}_t(y)\right)\,\mu(d\eta) \leq c_1\min\left(e^{-c_2(\norm{x-y}_\infty-c_3t)},e^{-\alpha t}\right).
\end{align}
\end{lemma}
\begin{proof}
The bound for $\norm{x-y}_\infty$ large relies on a comparison with first passage percolation. Consider first passage percolation with independent exponentially distributed edge weights with parameter $\lambda$. Let $T(x,y)$ be the first passage percolation travel distance between $x$ and $y$, i.e., the minimal weight of paths between $x$ and $y$. Let $B_t =\{z\in\bZ^d : T(x,z)\leq t\}$ be the ball of radius $t$ with respect to the distance $T$. By using the graphical construction of the contact process we can see that $B_t$ has the same distribution as the contact process at time $t$ started in $x$ if there are no recovery events. Hence $B_t$ stochastically dominates $\xi_t^x$. Note that new discrepancies between $\xi_t^{\eta^x}$ and $\xi^\eta_t$ can only arise via infection events from a site in $\xi^\eta_t \Delta\xi^{\eta^x}_t$, where the $\Delta$ is the symmetric difference. Hence $\xi^\eta_t \Delta\xi^{\eta^x}_t \subset \xi^x_t$, and
\begin{align}\label{eq:secondclass-a}
 \bP\left(\xi^\eta_t(y)\neq\xi^{\eta^x}_t(y)\right) &=  \bP\left(y\in \xi^\eta_t \Delta\xi^{\eta^x}_t \right) 
\leq \bP\left(y\in \xi^x_t\right) \leq \bP(T(x,y)\leq t).
\end{align}
From the theory of first passage percolation(see \cite{KESTEN:86}, Theorems 3.10, 3.11) we use the following fact:
There exist positive constants $c_2, c_3, c_4$ such that for all $x\in \bZ^d$ with $\norm{x}_\infty>c_3 t$,
\begin{align}
\bP(T(0,x) \leq t) \leq c_4 e^{-c_2 \norm{x}_\infty}. \label{eq:fpp1}	
\end{align}
From this and \eqref{eq:secondclass-a} follows 
\begin{align}
\int \bP\left(\xi^\eta_t(y)\neq\xi^{\eta^x}_t(y)\right)\,\mu(d\eta) \leq c_4e^{-c_2(\norm{x-y}_\infty-c_3t)}.
\end{align}

To prove 
\begin{align}
\int \bP\left(\xi^\eta_t(y)\neq\xi^{\eta^x}_t(y)\right)\,\mu(d\eta) \leq c_1e^{-\alpha t}
\end{align}
we use the self duality of the contact process. Since $\xi_t^\eta(y)=1$ corresponds to the event $\xi_t^y\cap \eta \neq \emptyset$, there is a difference in $y$ if and only if either $\xi_t^y\cap \eta$ or $\xi_t^y\cap \eta^x$ is empty. Hence
\begin{align}
 \bP\left(\xi^\eta_t(y)\neq\xi^{\eta^x}_t(y)\right) 
&= \bP\left(\xi^y_t \cap (\eta\cup\{x\}) = \{x\}\right)	\\
&= \bP\left((\xi^y_t\bs\{x\}) \cap \eta = \emptyset, x\in\xi^y_t\right)	\\
&\leq \bP\left((\xi^y_t\bs\{x\}) \cap \eta = \emptyset, \xi^y_t\neq\emptyset\right).
\end{align}
To estimate this probability, we can distinguish between survival or extinction of $\xi^y$ to obtain the upper bound
\begin{align}\label{eq:extinctionsurvival}
\bP\left((\xi^y_t\bs\{x\}) \cap \eta = \emptyset , \text{$\xi^y$ survives} \right) + \bP\left(\xi^y_t\neq\emptyset , \text{$\xi^y$ becomes extinct} \right).
\end{align}
We will treat the first and second term separately. If the initial infection at $y$ survives, Lemma \ref{lemma:ldp-survival} shows that typically $|\xi^y_t|$ is at least of order $t$. For the intersection $(\xi^y_t\bs\{x\}) \cap \eta$ to be empty the $\eta$ must be in some sense exceptional. Consider an arbitrary finite set $A\subset\bZ^d$. Since $\eta$ is $\mu$-distributed we can use self-duality and Lemma \ref{lemma:liggett} to get  constants $c_5,c_6>0$ independent of $A$ so that 
\begin{align}\label{eq:extinction}
\mu\left(\eta\cap A=\emptyset\right) = \bP(\xi^A \text{ becomes extinct}) \leq c_5e^{-c_6\abs{A}}
\end{align}
Conditioning on survival and integrating with respect to $\mu$ we obtain 
\begin{align}
&\int \bP\left((\xi^y_t\bs\{x\}) \cap \eta = \emptyset \;\middle|\; \xi^y \text{ survives}\right)\,\mu(d\eta) \\
&\quad\leq \bE \left(c_5 e^{-c_6(\abs{\xi_t^y}-1)}\;\middle|\;\xi^y \text{ survives}\right).
\end{align}
By Lemma \ref{lemma:ldp-survival} this is exponentially small.

To finish the proof we show that the second term in \eqref{eq:extinctionsurvival} is also exponentially small. We use the fact that an extinction happens early, or more precisely Lemma \ref{lemma:liggett}, from which we get
\begin{align}
 \bP(\xi^y_t \neq \emptyset \;|\; \xi^y \text{ becomes extinct}) \leq c_5e^{-c_6t}. 
\end{align}
\end{proof}
We have just proven that the probability for the event that an initial discrepancy at $x$ causes a difference at site $y$ and time $t$ decays exponentially fast with respect to the spatial and temporal distance. We use this to prove that the second moment of the total number of discrepancies at time $t$ is small.
\begin{lemma}\label{lemma:cluster}
There is a constant $c>0$ so that
\begin{align}
\int \bE \abs{\xi_t^{\eta^x}\Delta \xi_t^\eta}^2\mu(d\eta)\leq c(t+1)^{3d}e^{-\alpha t}.
\end{align}
\end{lemma}
\begin{proof}
Consider the smallest cube centred at the origin which contains $\xi_t^{\eta^x}\Delta \xi_t^\eta$. If its radius is $\widetilde R$, then clearly $\abs{\xi_t^{\eta^x}\Delta \xi_t^\eta}^2 \leq (2\widetilde{R}+1)^{2d}$. By minimality of the cube there is an element of $\xi_t^{\eta^x}\Delta \xi_t^\eta$ on its boundary. Hence,
\begin{align}
&\int \bE \abs{\xi_t^{\eta^x}\Delta \xi_t^\eta}^2\mu(d\eta) \\
&\quad\leq \sum_{R=0}^\infty (2R+1)^{2d} \int\bP\left(\exists\, y:\norm{x-y}_\infty =R, \xi_t^{\eta^x}(y)\neq\xi_t^\eta(y) \right)\,\mu(d\eta).	 \label{eq:cluster-1}
\end{align}
By Lemma \ref{lemma:discrepancy}, using the same constants,
\begin{align}
&\int \bP\left(\exists y:\norm{x-y}_\infty =R, \xi_t^{\eta^x}(y)\neq\xi_t^\eta(y) \right) \mu(d\eta) \\
&\quad\leq \sum_{y:\norm{x-y}_\infty =R} \int \bP\left(\xi_t^{\eta^x}(y)\neq\xi_t^\eta(y) \right) \mu(d\eta) 	\\
&\quad\leq 2d(2R+1)^{d-1} c_1\min\left(e^{-c_2(R-c_3t)},e^{-\alpha t}\right).  
\end{align}
Hence, by splitting the sum at some value $R_t$ to be determined later,
\begin{align}
 \eqref{eq:cluster-1}\leq &(R_t+1)(2R_t+1)^{2d}2d(2R_t+1)^{d-1}c_1e^{-\alpha t}	\\
 &+ \sum_{R=1}^\infty (2R+2R_t+1)^{2d}2d(2R+2R_t+1)^{d-1}c_1e^{-c_2(R+R_t-c_3t)}	\\
\leq& 2dc_1 (2R_t+1)^{3d}e^{-\alpha t}
+ 2dc_1 e^{-c_2R_t+c_2c_3 t}\sum_{R=1}^\infty (2R+2R_t+1)^{3d-1}e^{-c_2 R}.\label{eq:cluster-2}
\end{align}
Choosing $R_t=\lfloor (\alpha/c_2+c_3)t+1\rfloor$ and using the fact that $(2R+1+2R_t)^{3d-1}\leq 2^{3d-1}(2R+1)^{3d-1}(2R_t)^{3d-1}$, we can find a suitable constant $c>0$ so that
\begin{align}
 \eqref{eq:cluster-2} &\leq c (t+1)^{3d}e^{-\alpha t}. \qedhere
\end{align} 
\end{proof}
\begin{proof}[Proof of Proposition \ref{prop:var-decay}]
By using the fact that $\int L [P_tf \overline{P_tf}] \,d\mu = 0$,
\begin{align}
 \frac{d}{dt}\Var_\mu(P_t f) = -\int \left(L\left[(P_tf-P_tf(\eta))\overline{(P_tf-P_tf(\eta))}\right]\right)(\eta)\,\mu(d\eta).
\end{align}
Together with $\lim_{t\to\infty}\Var_\mu(P_tf)=0$, this gives us the variance estimate
\begin{align}
\Var_\mu(P_t f) &= \int_t^\infty \int L\left[(P_s f-P_sf(\eta))\overline{(P_s f-P_sf(\eta))}\right](\eta)\,\mu(d\eta)ds	\\
 &\leq (2d\lambda+1) \int_t^\infty  \int \sum_{x\in\bZ^d} \abs{P_sf(\eta^x) -P_sf(\eta)}^2\,\mu(d\eta)ds.	\label{eq:var-2}
\end{align}
Estimating the inner term by considering the maximal influences $\delta_f$ (remember \eqref{eq:deltas}) at the discrepancies, applying Jensen's inequality and then using translation invariance of the contact process,
\begin{align}
\abs{P_sf(\eta^x) -P_sf(\eta)}^2 
&\leq \left(\bE\sum_{y\in\bZ^d}\ind_{\xi^\eta_s(y)\neq\xi_s^{\eta^x}(y)}\delta_f(y) \right)^2	\\
&\leq \bE \abs{\xi^\eta_s\Delta\xi^{\eta^x}}\sum_{y\in\bZ^d} \ind_{\xi^\eta_s(y)\neq\xi_s^{\eta^x}(y)}(\delta_f(y))^2	\\
&= \bE \abs{\xi^\eta_s\Delta\xi^{\eta^0}}\sum_{y\in\bZ^d} \ind_{\xi^\eta_s(y)\neq\xi_s^{\eta^0}(y)}(\delta_f(x+y))^2.
\end{align}
Hence
\begin{align}
 \eqref{eq:var-2} &\leq (2d\lambda+1) \int_t^\infty  \int \bE \abs{\xi^\eta_s\Delta\xi^{\eta^0}}^2\mu(d\eta)ds \sum_{x\in\bZ^d}\delta_f(x)^2.
\end{align}
The claim then follows by Lemma \ref{lemma:cluster}.
\end{proof}

\section{Contact process in a finite volume}\label{section:finite}
In this section we consider the contact process in a finite region, $\Lambda_N=\{-N,...,N\}^d$, with a boundary condition of all sites infected. Let $\Omega_N=\{\eta\in\Omega: \eta\equiv 1 \text{ off } \Lambda_N\}$. The generator of the contact process on $\Omega_N$ is given by
\begin{align}
 L_N f(\eta) = \sum_{x\in\Lambda_N}\eta(x)[f(\eta^x)-f(\eta)] + \sum_{x\in\Lambda_N}\sum_{y\sim x}\lambda \eta(y)(1-\eta(x))[f(\eta^x)-f(\eta)].
\end{align}
Let $P_{N,t}$ be the corresponding semi-group and $\mu_N$ the invariant measure on $\Omega_N$.

The corresponding Markov process $(\xi_{N,t})_{t\geq0}$ can be constructed from the same graphical construction as the infinite process on $\bZ^d$, simply by ignoring all recovery events outside of $\Lambda_N$. Therefore we write $(\xi_{N,t})_{t\geq0}$ also for the contact process generated by 
\begin{align}
 \widetilde{L}_N f(\eta) = \sum_{x\in\Lambda_N}\eta(x)[f(\eta^x)-f(\eta)] + \sum_{x\in\bZ^d}\sum_{y\sim x}\lambda \eta(y)(1-\eta(x))[f(\eta^x)-f(\eta)],
\end{align}
which coincides with the finite contact process on $\Omega_N$.

The finite volume contact process $\xi_N$ has a similar dual relation as the infinite contact process. Since the finite contact process differs from the infinite one only by the fact that infections outside $\Lambda_N$ survive forever we can relate the two. For $\eta\in\Omega_N,A\subset \Lambda_N$,
\begin{align}\label{eq:finite-duality}
 \bP(\xi^\eta_{N,t}\cap A \neq \emptyset) &= \bP(\xi^A_{N,t} \cap \eta \neq \emptyset ) = \\
&=\bP\left(\xi_t^A \cap \eta \neq \emptyset \text{ or }\exists\,s\in[0,t]: \xi_s^A \cap \Lambda_N^c \neq \emptyset\right) .
\end{align}
Note how on the right hand side the contact process on the infinite lattice with initial configuration $A\in\Omega$ is used. The graphical construction also allows for a natural coupling of $\xi$ and $\xi_N$.

Initially we will reprove the results from Section \ref{section:Zd} in the finite setting. Mostly the proves do not change much from the infinite case. Unless otherwise mentioned, constants are the same as in the corresponding statements in Section \ref{section:Zd}.
\begin{lemma}\label{lemma:finite-discrepancy}
 There are constants $c_1,c_2,c_3,\alpha>$ so that for all $N>0$, $\eta\in\Omega_N$ and $x,y\in\Lambda_N$
\begin{align}
\int \bP\left(\xi^\eta_{N,t}(y)\neq\xi^{\eta^x}_{N,t}(y)\right)\,\mu(d\eta) \leq c_1\min\left(e^{-c_2(\norm{x-y}_\infty-c_3t)},e^{-\alpha t}\right).
\end{align}
\end{lemma}
\begin{proof}
 The part for large $\norm{x-y}_\infty$ is the same as in Lemma \ref{lemma:discrepancy} by the fact that $\xi_{N,t}^{\eta^x}\Delta \xi_{N,t}^\eta \subset \xi_t^x$.

The large $t$ estimate uses \eqref{eq:finite-duality}, from which it follows that
\begin{align}
 \bP\left(\xi^{\eta^x}_{N,t}(y) \neq \xi^\eta_{N,t}(y) \right) &= \bP\left(\xi_t^y \cap (\eta\cup \{x\}) = \{x\}, \forall\,s\in[0,t]: \xi_s^y \cap \Lambda_N^c = \emptyset\right) \\
&\leq \bP\left(\xi_t^y \cap (\eta\cup \{x\}) = \{x\} \right).
\end{align}
From here the proof continues as in the infinite volume case.
\end{proof}

\begin{lemma}\label{lemma:finite-cluster}
There is a constant $c>0$ so that for all $N>0$, 
\begin{align}
\int \bE \abs{\xi_{N,t}^{\eta^x}\Delta \xi_{N,t}^\eta}^2\mu(d\eta)\leq c(t+1)^{3d}e^{-\alpha t}.
\end{align}
\end{lemma}
\begin{proof}
 The proof is identical to the one of Lemma \ref{lemma:cluster} by virtue of Lemma \ref{lemma:finite-discrepancy}.
\end{proof}

\begin{proposition}\label{prop:finite-var-decay}
\begin{align}
\Var_{\mu_N}(P_{N,t}f)\leq C \int_t^\infty (s+1)^{3d} e^{-\alpha s}\,ds \sum_{x\in\Lambda_N}\delta_f(x)^2.
\end{align}
\end{proposition}
The proof is identical to Proposition \ref{prop:var-decay}.

\begin{proposition}\label{prop:finite-gap}
 The generator $L_N$ has a spectral gap of at least $\alpha/2$.
\end{proposition}
\begin{proof}
 Since $\Omega_N$ is finite $-L_N$ can be written as a finite matrix. Its spectrum consists of finitely many eigenvalues. By irreducibility the eigenvalue 0 has multiplicity 1, with the constant functions as eigenfunctions. Consider an eigenvalue $\lambda\neq0$ with eigenfunction $f:\Omega_N\to\bC$. Since $P_{N,t}f=e^{-\lambda t}f$ we have
\begin{align}
 \Var_{\mu_N}(P_{N,t}f) = e^{-2\Re(\lambda)t}\Var_{\mu_N}(f),
\end{align}
and by Proposition \ref{prop:finite-var-decay},
\begin{align}
 e^{-2\Re(\lambda)t}\Var_{\mu_N}(f) \leq C \int_t^\infty (s+1)^{3d} e^{-\alpha s}\,ds \sum_{x\in\Lambda_N}\delta_f(x)^2
\end{align}
holds for all $t\geq0$. This in turn implies $\Re(\lambda)\geq \alpha/2$.
\end{proof}

\section{Spectral gap on the infinite lattice}\label{section:gap}
What remains to do is to extend the spectral gap from the finite system to the infinite lattice. Since the estimate of the gap is uniform in $N$ this is straight forward, and consists of showing that $\mu_N\to\mu$ and $P_{N,t}\to P_t$ in a suitable sense. The first step is to show that inside a large but fixed box $\Lambda_L$ the measures $\mu_N$ converge to $\mu$.

\begin{lemma}\label{lemma:total-variation}
 Let $\mu|_{\Lambda_L}$ and $\mu_N|_{\Lambda_L}$ be the restrictions of the measures to $\{0,1\}^{\Lambda_L}$. For any $L>0$,
\begin{align}
 \lim_{N\to\infty} \norm{\mu|_{\Lambda_L} - \mu_N|_{\Lambda_L}}_{TV} =0,
\end{align}
where the norm is the total variation distance.
\end{lemma}
\begin{proof}
Fix $L>0$ and assume $N>3L$.
 Define for a set $A\subset \Lambda_L$
\begin{align}
 \tau(A) &:= \inf\left\{t\geq 0 : \xi_t^A = \emptyset \right\},	\\
 \sigma_N(A) &:= \inf\left\{t\geq 0 : \xi_t^A \cap \Lambda_N^c \neq \emptyset \right\},	\\
\tau_N(A) &:= 
\begin{cases}
\tau(A),\quad&\sigma_N(A)=\infty;\\
\infty, &\sigma_N(A)<\infty.               
\end{cases}
\end{align}
We use the usual convention that the infimum of the empty set is $\infty$. The stopping times defined above are the extinction time, hitting time of the boundary, and extinction time before hitting the boundary respectively. If the set $A$ is a singleton $\{a\}$, we simply write $\tau(a),\sigma_N(a), \tau_N(a)$.
Remember the comparison of the contact process with first passage percolation in Lemma \ref{lemma:discrepancy} and estimate \eqref{eq:fpp1}.
We have constants $c_1, c_2, c_3>0$ so that
\begin{align}\label{eq:fpp3}
 \bP( y \in \xi_{c_1 N}^{a} )\leq \bP(T(a,y)\leq c_1N) \leq c_2e^{-c_3N}
\end{align}
for all $a\in\Lambda_L, y\in \Lambda_N^c$. 
Let 
\[ \cB= \cB_N = \left\{\forall a\in \Lambda_L : \sigma_{N}(a) > c_1 N \right\}. \]
By \eqref{eq:fpp3} we have that $\lim_{N\to\infty}\bP(\cB_N^c)=0$.

Fix $A\subset\Lambda_L$. We want to prove that the probability of the event $\{\eta \cap \Lambda_L = A \}$ is almost the same under $\mu$ and under $\mu_N$, with a vanishing difference. To do so we use duality. Remember that by duality a site $x$ is infected under $\mu$ if and only if in the dual representation $\xi^x$ survives. Therewith the event $\{\eta \cap \Lambda_L = A \}$ can be reformulated via survival and extinction events in the dual formulation: Each $a\in A$ must survive, and each $a \in \Lambda_L\bs A$ must become extinct. This leads to
\begin{align}\label{eq:dual1}
 \int \ind_{\{\eta \cap \Lambda_L = A\}} \mu({d\eta}) = \bP\left(\forall a\in A: \tau(a)=\infty, \tau(\Lambda_L\bs A) <\infty\right).
\end{align}
For $\mu_N$ we have essentially the same argument, with the difference that an infection can also come from the boundary. Hence
\begin{align}\label{eq:dual2}
 \int \ind_{\{\eta \cap \Lambda_L = A\}} \mu_N({d\eta}) = \bP\left(\forall a\in A: \tau_N(a)=\infty, \tau_N(\Lambda_L\bs A) <\infty \right).
\end{align}
By Lemma \ref{lemma:liggett}, there are constants $c_4, c_5>0$ so that
\begin{align}\label{eq:extinction-estimate}
 \bP(c_1 N < \tau(\Lambda_L\bs A) < \infty) \leq c_4 e^{-c_5 N}.
\end{align}
By definition we have $\tau\leq \tau_N$, and hence
\begin{align}
 &\{\forall a\in A: \tau_N(a)=\infty \} \bs \{\forall a\in A: \tau(a)=\infty \} \\
&\quad= \{\exists a\in A: \tau(a)<\infty, \sigma_N(a) <\infty \}.
\end{align}
On the event $\cB$ we have 
\[ \{ \tau(\Lambda_L\bs A) \leq c_1 N \}=\{ \tau_N(\Lambda_L\bs A) \leq c_1 N \}. \]
Hence, using the formulations \eqref{eq:dual1} and \eqref{eq:dual2} and by focusing on the good events $\cB$ and $ \{ \tau(\Lambda_L\bs A) \leq c_1 N \}$, 
\begin{align}
 &\abs{\int \ind_{\{\eta \cap \Lambda_L = A\}} (\mu-\mu_N)({d\eta})} \\
&\leq \bP\left(\cB, \tau(\Lambda_L\bs A) \leq c_1 N, \exists a\in A: \tau(a)<\infty, \sigma_N(a)<\infty \right) \\
&\quad+ \bP(\cB, c_1N<\tau(\Lambda_L\bs A)<\infty) + \bP(\cB^c)	\\
&\leq \sum_{a\in A} \bP(\cB, \tau(a)<\infty, \sigma_N(a)<\infty)+ \bP(c_1N< \tau(\Lambda_L\bs A)<\infty) + \bP(\cB^c).
\end{align}
By \eqref{eq:extinction-estimate} and \eqref{eq:fpp3}, the last two terms go to 0 as $N\to\infty$. For the first term,
\begin{align}
 \bP(\cB, \tau(a)<\infty, \sigma_N(a)<\infty) & \leq \bP(\tau(a) < \infty, c_1N<\sigma_N(a)<\infty)	\\
&\leq \bP(c_1 N < \tau(a) < \infty),
\end{align}
which goes to 0 by \eqref{eq:extinction-estimate}. Hence
\[ \lim_{N\to\infty}\abs{\int \ind_{\{\eta \cap \Lambda_L = A\}} (\mu-\mu_N)({d\eta})}=0 \]
for all of the finitely many $A\subset \Lambda_L$, which proves the claim.
\end{proof}

\begin{lemma}\label{lemma:approximation}
 Let $f:\Omega \to\bC$ be a local function and $t\geq 0$. Then
\begin{align}
 \lim_{N\to\infty}\abs{\Var_\mu(P_t f )-\Var_{\mu_N}(P_{N,t} f)} = 0.
\end{align}
\end{lemma}
\begin{proof}
Without loss of generality assume $\int f \,d\mu =0$. We have
\begin{align}
 &\abs{\Var_\mu(P_t f )-\Var_{\mu_N}(P_{N,t} f)} \\
&\leq \abs{\int \abs{P_tf}^2 d(\mu-\mu_N)}+\abs{\int \abs{P_tf}^2-\abs{P_{N,t}f}^2d\mu_N}+\abs{\int f \,d\mu_N}^2. \label{eq:approximation-1}
\end{align}
We consider the terms separately and start with the middle one. For any $\eta\in\Omega_N$, we have
\begin{align}
 \abs{\abs{P_tf(\eta)}^2-\abs{P_{N,t}f(\eta)}^2} &\leq \abs{P_tf(\eta)-P_{N,t}f(\eta)} 2\norm{f}_\infty	\\
&\leq \bP\left (\exists\, x \in \supp(f) : \xi^\eta_t(x) \neq \xi^\eta_{N,t}(x)\right) 4 \norm{f}^2_\infty .
\end{align}
 Here we can use the graphical construction.  We see that for $\xi^\eta_t$ and $\xi^\eta_{N,t}$ to differ inside $\supp(f)$ there must be an infection path from the boundary of $\Lambda_N$ to a site in $\supp(f)$. By reversing time and looking at the dual process, we get an upper bound independent of $\eta$:
\begin{align}
 \bP\left (\exists\, x \in \supp(f) : \xi^\eta_t(x) \neq \xi^\eta_{N,t}(x) \right) \leq \bP\left(\exists\, s\in[0,t]: \xi_s^{\supp(f)}\not\subset\Lambda_N \right).
\end{align}
Hence the middle term of \eqref{eq:approximation-1} converges to 0 as $N\to\infty$.
The last term goes to 0 directly as a consequence of Lemma \ref{lemma:total-variation}, since $f$ is local with $\int f \,d\mu =0$. The first term also vanishes as a consequence of Lemma \ref{lemma:total-variation} after we approximate $(P_t f)^2$ by local functions.
\end{proof}

\begin{theorem}
 The supercritical contact process on $\bZ^d$ has a spectral gap of at least $\alpha/2$.
\end{theorem}
\begin{proof}
 Let $f:\Omega\to\bC$ local. Using the spectral gap for the finite system, Proposition \ref{prop:finite-gap}, for any $N$,
\begin{align}
 \Var_\mu(P_t f) &= \Var_\mu(P_tf) - \Var_{\mu_N}(P_{N,t} f) + \Var_{\mu_N}(P_{N,t} f)	\\
&\leq \Var_\mu(P_tf) - \Var_{\mu_N}(P_{N,t} f) + e^{-\alpha t} \Var_{\mu_N}(f)	\\
&= e^{-\alpha t} \Var_{\mu}(f) + \Var_\mu(P_tf) - \Var_{\mu_N}(P_{N,t} f) \\
&\quad + e^{-\alpha t}\left(\Var_{\mu_N}(f)-\Var_{\mu}(f)\right).
\end{align}
When sending $N$ to infinity, we have by virtue of Lemma \ref{lemma:approximation}
\begin{align}
 \Var_\mu(P_t f) &\leq e^{-\alpha t} \Var_{\mu}(f)
\end{align}
Approximating $f\in L^2(\mu)$ by local functions completes the proof.
\end{proof}

\bibliography{BibCollection}{}
\bibliographystyle{plain}

\end{document}